\documentclass[11pt,leqno]{article}
\usepackage{exscale,relsize}
\usepackage{amsmath}
\usepackage{amsfonts}
\usepackage{amssymb}
\usepackage{calc}
\usepackage{theorem}
\usepackage{pifont}      
\usepackage{xcolor}
\newcommand{\emp}{\ensuremath{\varnothing}}

\newcommand{\scal}[2]{\langle{{#1},{#2}}\rangle}

\newcommand{\exi}{\ensuremath{\exists\,}}

\newcommand{\RR}{\ensuremath{\mathbb R}}

\newcommand{\RX}{\ensuremath{\,\left]-\infty,+\infty\right]}}

\newcommand{\menge}[2]{\big\{{#1} \mid {#2}\big\}}

\newcommand{\cone}{\ensuremath{\operatorname{cone}}}
\newcommand{\dom}{\ensuremath{\operatorname{dom}}}

\newcommand{\gra}{\ensuremath{\operatorname{gra}}}

\newcommand{\inte}{\ensuremath{\operatorname{int}}}

\newcommand{\ran}{\ensuremath{\operatorname{ran}}}

\newcommand{\kernel}{\ensuremath{\operatorname{ker}}}

\newcommand{\halo}{\ensuremath{\operatorname{halo}}}

\newcommand{\pinf}{\ensuremath{+\infty}}

\renewcommand{\phi}{\ensuremath{\varphi}}

\newcommand{\To}{\ensuremath{\rightrightarrows}}

\newtheorem{theorem}{Theorem}[section]
\newtheorem{lemma}[theorem]{Lemma}
\newtheorem{fact}[theorem]{Fact}
\newtheorem{corollary}[theorem]{Corollary}
\newtheorem{proposition}[theorem]{Proposition}
\newtheorem{definition}[theorem]{Definition}

\theoremstyle{plain}{\theorembodyfont{\rmfamily}
}
\theoremstyle{plain}{\theorembodyfont{\rmfamily}
}
\theoremstyle{plain}{\theorembodyfont{\rmfamily}
}
\theoremstyle{plain}{\theorembodyfont{\rmfamily}
\newtheorem{example}[theorem]{Example}}
\theoremstyle{plain}{\theorembodyfont{\rmfamily}
\newtheorem{remark}[theorem]{Remark}}
\theoremstyle{plain}{\theorembodyfont{\rmfamily}
}

\begin{document}

\title{{\sffamily Monotone Linear Relations:\\
Maximality and Fitzpatrick Functions}}

\author{
Heinz H.\ Bauschke\thanks{Mathematics, Irving K.\ Barber School,
UBC Okanagan, Kelowna, British Columbia V1V 1V7, Canada. E-mail:
\texttt{heinz.bauschke@ubc.ca}.},~  Xianfu
Wang\thanks{Mathematics, Irving K.\ Barber School, UBC Okanagan,
Kelowna, British Columbia V1V 1V7, Canada. E-mail:
\texttt{shawn.wang@ubc.ca}.},~ and ~ Liangjin\
Yao\thanks{Mathematics, Irving K.\ Barber School, UBC Okanagan,
Kelowna, British Columbia V1V 1V7, Canada.
E-mail:  \texttt{ljinyao@interchange.ubc.ca}.} }

\date{
May 27, 2008 
\vskip 1 cm
\emph{Dedicated to Stephen Simons on the occasion of his $70^\text{th}$
birthday}
}

\maketitle

\begin{abstract} \noindent
We analyze and characterize maximal monotonicity of  linear relations
(set-valued operators with linear graphs). An important tool in our
study are
Fitzpatrick functions. The results obtained partially extend work on
linear and at most single-valued operators by
Phelps and Simons and by
Bauschke, Borwein and Wang.
Furthermore, a description of skew linear relations
in terms of the Fitzpatrick family is obtained.
We also answer one of Simons' problems by showing that if 
a maximal monotone operator has a convex graph, then this graph must actually be affine.
\end{abstract}

\noindent {\bfseries 2000 Mathematics Subject Classification:}\\
Primary 47A06, 47H05; Secondary 26B25, 47A05, 49N15, 52A41, 90C25.

\noindent {\bfseries Keywords:} \\
Adjoint process,
Fenchel conjugate,
Fitzpatrick family,
Fitzpatrick function,
linear relation,
maximal monotone operator,
monotone operator,
skew linear relation.

\section{Introduction}
Linear relations (also known as linear processes)
have been considered by many authors for a long time;
see \cite{aubin,Cross,borwein1,arens} and the many references therein.
Surprisingly, the class of monotone (in the sense of set-valued
analysis)
linear relations has been explored much less even
though it provides a considerably
broader framework for studying monotone linear operators and
its members arise frequently in optimization,
functional analysis and functional equations.

This paper focuses on
monotone linear relations, i.e., on monotone operators with linear graphs.
We discuss the relationships of domain, range, and kernel between original
and adjoint linear relation as well as Fitzpatrick functions
and criteria for maximal monotonicity.
Throughout, $X$ denotes a reflexive real Banach
space, with continuous dual space $X^*$, and with pairing
$\scal{\cdot}{\cdot}$.
The Notation used is standard and as in Convex Analysis and Monotone
Operator Theory; see, e.g., \cite{ph,Rocky,RockWets,Si,Zalinescu}.
Let $A$ be a set-valued operator (also known as multifunction)
from $X$ to $X^*$.
Then the
\emph{graph} of $A$ is
$\gra A := \menge{(x,x^*)\in X\times X^*}{x^*\in Ax}$, and $A$  is
\emph{monotone} if
\begin{equation*}
\big(\forall (x,x^*)\in \gra A\big)\big(\forall (y,y^*)\in\gra
A\big) \quad \scal{x-y}{x^*-y^*}\geq 0.
\end{equation*}
$A$ is said to be \emph{maximal monotone} if no proper enlargement
(in the sense of graph inclusion) of $A$ is monotone.
The \emph{inverse
operator} $A^{-1}\colon X^*\To X$ is given by $\gra A^{-1} :=
\menge{(x^*,x)\in X^*\times X}{x^*\in Ax}$; the \emph{domain} of $A$ is
$\dom A := \menge{x\in X}{Ax\neq\varnothing}$; its
\emph{kernel} is $\ker A:= \menge{x\in
X}{0\in Ax}$, and its \emph{range} is $\ran A: = A(X)$. We say
$(x,x^*)\in X\times X^*$ is \emph{monotonically related} to $\gra A$ if
$(\forall (y,y^*)\in\gra A)$
$\scal{x-y}{x^*-y^*}\geq 0$.
The \emph{adjoint} of $A$, written $A^*$, is defined
by
\begin{equation*}
\gra A^* :=
\menge{(x,x^*)\in X\times X^*}{(x^*,-x)\in (\gra A)^\bot},
\end{equation*}
where $(\gra A)^{\bot}=\menge{(y^*,y)\in X^*\times X}{(\forall (a,a^*)\in\gra A)\;\; 
\langle y^*,a\rangle+\langle a^*,y\rangle =0}.$
We say $A$
is a maximal monotone linear relation if $A$ is a maximal monotone operator
and $\gra A$ is a linear subspace of $X\times X^*$.
Finally,
if $f\colon X\rightarrow\RX$ is proper and convex, we write
$f^* \colon x^*\mapsto \sup_{x\in X} \scal{x}{x^*}-f(x)$,
and
$\dom f := \menge{x\in X}{f(x)<\pinf}$,
for the \emph{Fenchel conjugate},
and the \emph{domain} of $f$, respectively.

The outline of the paper is as follows. In Section~\ref{s:2} we provide
preliminary results about monotone linear relations. In
Section~\ref{s:3}, the relationships between domains, ranges, and kernels of $A$
and $A^*$ are discussed. In Section~\ref{s:4}, we present a
result that states that a maximal monotone operators
with convex graphs must be affine.   Section~\ref{s:5} provides useful
relationships among the Fitzpatrick functions $F_{A+B}$, $F_{A}$ and
$F_{B}$.  In Section~\ref{s:6},
the maximality criteria for monotone linear relations
are established; these generalize corresponding results by
by Phelps and Simons \cite{PheSim} on linear (at most single-valued)
monotone operators.  The final Section~\ref{s:7}
contains a characterization of
skew linear operators in terms of the single-valuedness of
Fitzpatrick family associated to the monotone operator.
The results in Sections \ref{s:3}, \ref{s:4} and \ref{s:7}
extend their single-valued counterparts in \cite{BBW}
to monotone linear relations.

\section{Auxiliary Results for Monotone Linear Relations}
\label{s:2}

\begin{fact}\label{Rea:1}
Let $A \colon X \To X^*$ be a linear relation.
Then the following hold.
\begin{enumerate}
\item \label{Th:26} $A0$ is a linear subspace of $X^*$.
\item \label{Th:28} $(\forall (x,x^*)\in\gra A)$
$Ax=x^* +A0$.
\item \label{Th:30}
$(\forall x\in\dom A)(\forall y\in\dom A)(\forall (\alpha,\beta)
\in\RR^2\smallsetminus\{0,0\})$
$A(\alpha x+\beta y) = \alpha Ax+\beta Ay$.
\item \label{Sia:2b}
$(\forall x\in \dom A^*)(\forall y\in\dom A)$
$\scal{A^*x}{y}=\scal{x}{Ay}$ is a singleton.
\item \label{Sia:1}$ \overline{\dom A}=(A^*0)^{\perp}$.
If $\gra A$ is closed, then $(\ker A)^{\bot}=\overline{\ran A^*}$,
$\overline{\dom A^*}= (A0)^{\bot}$,
 and  $A^{**} = A$.
\end{enumerate}
\end{fact}
\begin{proof}
\ref{Th:26}: See \cite[Corollary~I.2.4]{Cross}.
\ref{Th:28}: See \cite[Proposition~I.2.8(a)]{Cross}.
\ref{Th:30}: See \cite[Corollary~I.2.5]{Cross}.
\ref{Sia:2b}: See \cite[Proposition~III.1.2]{Cross}.
\ref{Sia:1}: See \cite[Proposition~III.1.4(b)--(d), Theorem~III.4.7, and Exercise~VIII.1.12]{Cross}.
\end{proof}

\begin{proposition}\label{P:a}
Let $A\colon X \To X^*$ be a monotone linear relation.
Then the following hold.
\begin{enumerate}
\item\label{Th:27}
$\dom A\subset (A0)^\bot$ and $A0\subset (\dom A)^\bot$;
consequently, if $\gra A$ is closed, then
$\dom A\subset\overline{\dom A^*}$ and $A0\subset A^*0$.
\item \label{Th:32}
$(\forall x\in\dom A)(\forall z\in (A0)^\bot)$
$\langle z,Ax\rangle$ is single-valued.
\item \label{Th:32+}
$(\forall z \in (A0)^\bot)$ $\dom A \to \RR \colon y\mapsto \scal{z}{Ay}$
is linear.
\item \label{Th:31} $A$ is monotone $\Leftrightarrow$
$(\forall x\in\dom A)$ $\langle x,Ax\rangle$ is single-valued and
$\langle x,Ax\rangle\geq 0$.
\item \label{Th:31+} If $(x,x^*)\in (\dom A)\times X^*$ is monotonically
related to $\gra A$ and $x_{0}^*\in Ax$, then
$x^*-x_{0}^*\in (\dom A)^{\perp}.$
\end{enumerate}
\end{proposition}
\begin{proof}
\ref{Th:27}: Pick $x\in \dom A$.
Then there exists $x^*\in X$ such that
  $(x,x^*)\in \gra A.$  By
  monotonicity of $A$ and since $(0,0)\in \gra A$,
 we have  $\langle x, x^*\rangle\geq \langle x,A0\rangle$.
Since $A0$ is a linear subspace (Fact~\ref{Rea:1}\ref{Th:26}),
we obtain
$x\bot A0$. This implies
$\dom A\subset (A0)^\bot$ and $A0\subset (\dom A)^\bot$.
As $\gra A$ is closed, Fact~\ref{Rea:1}\ref{Sia:1}
yields $\dom A\subset\overline{\dom A^*}$ and  $A0\subset A^*0$.

\ref{Th:32}:
Take $x\in\dom A$, $x^*\in Ax$, and $z\in (A0)^\bot$.
By Fact~\ref{Rea:1}\ref{Th:28},
$\langle z,Ax\rangle=\langle z,x^*+A0\rangle=\langle z,x^*\rangle$.

\ref{Th:32+}:
Take $z\in (A0)^{\perp}$. By \ref{Th:32},
$(\forall y\in\dom A)$ $\langle z, Ay\rangle$
is single-valued.
Now let $x,y$ be in $\dom A$, and let $\alpha,\beta$ be in $\RR$.
If $(\alpha,\beta) = (0,0)$, then
$\langle z, A(\alpha x+\beta y)\rangle=\langle z, A0\rangle=0
=\alpha \langle z, Ax\rangle +\beta \langle z, Ay\rangle$.
And if $(\alpha,\beta)\neq (0,0)$, then
Fact~\ref{Rea:1}\ref{Th:30} yields
$\langle z, A(\alpha x+\beta y)=\langle z, \alpha Ax +\beta Ay\rangle
=\alpha \langle z, Ax\rangle +\beta\langle z, Ay\rangle$.
This verifies the linearity.

\ref{Th:31}: ``$\Rightarrow$'': This  follows from
\ref{Th:27}, \ref{Th:32}, and the fact that $(0,0)\in\gra A$.
``$\Leftarrow$'': If $x$ and $y$ belong to $\dom A$, then
Fact~\ref{Rea:1}\ref{Th:30} yields
$\langle x-y,Ax-Ay\rangle=\langle x-y,A(x-y)\rangle\geq 0$.

\ref{Th:31+}: Let $(x,x^*)\in(\dom A)\times X^*$
be monotonically related to $\gra A$,
and take $x_0^*\in Ax.$
For every $(v,v^*)\in\gra A$, we have that $x_0^*+
v^*\in A(x+v)$ (by Fact~\ref{Rea:1}\ref{Th:30}); hence,
$\langle x-(x+v), x^*-(x_0^*+ v^*)\rangle\geq 0$ and thus
$\langle v,  v^*\rangle\geq \langle v, x^*-x_0^*
\rangle$. Now take
$\lambda>0$  and replace $(v,v^*)$ in the last inequality
by $(\lambda v,\lambda v^*)$.
Then divide by $\lambda$ and let $\lambda\rightarrow 0^+$ to see that
$0\geq \langle \dom A,\ x^*-x_0^*\rangle$.
Since $\dom A$ is linear, it follows that  $x^*-x_0^*\in
 (\dom A)^\bot$.
\end{proof}

For $A\colon X \To X^*$ it will be convenient
to define (as in, e.g., \cite{BBW})
\begin{equation*}
(\forall x\in X)\quad
q_A(x) :=  \begin{cases} \tfrac{1}{2}\langle x,Ax\rangle,&\text{if $x\in \dom A$};\\
\infty,&\text{otherwise}.\end{cases}\end{equation*}

\begin{proposition}\label{rP:1} Let $A\colon X \To X^*$ be a linear
relation,
let $x$ and $y$ be in $\dom A$, and let $\lambda\in\RR$.
Then
\begin{equation} \label{deli:1}
\lambda q_A(x) + (1-\lambda)q_A(y) - q_{A}(\lambda x + (1-\lambda)y) =
\lambda(1-\lambda)q_A(x-y) =
\tfrac{1}{2}\lambda(1-\lambda)\scal{x-y}{Ax-Ay}.
\end{equation}
Moreover,
$A$ is monotone $\Leftrightarrow$
$q_A$ is single-valued and convex.
\end{proposition}
\begin{proof}
Proposition~\ref{P:a}\ref{Th:27}\&\ref{Th:32} shows
that $q_A$ is single-valued on $\dom A$.
Combining with Proposition~\ref{P:a}\ref{Th:32}, we obtain
\eqref{deli:1}.
The characterization now follows from
Proposition~\ref{P:a}\ref{Th:31}.
\end{proof}

\begin{proposition}\label{j22}
Let $A\colon X \To X^*$ be a maximal monotone linear relation.
Then $(\dom A)^{\perp}= A0$ and hence
$\overline{\dom A}=(A0)^\bot$.
\end{proposition}

\begin{proof}
Since
$A+N_{\dom A}=A+(\dom A)^{\perp}$ is a monotone extension of $A$ and $A$ is
maximal monotone, we must have
$A+(\dom A)^{\perp}=A$. Then $A0+(\dom A)^{\perp}=A0$. As $0\in A0$,
$(\dom A)^{\perp}\subset A0.$ The reverse inclusion follows from
Proposition~\ref{P:a}\ref{Th:27}.
 \end{proof}

\section{Domain, Range, Kernel, and Adjoint}
\label{s:3}

In this section, we study relationships among domains,
ranges and kernels of a maximal monotone linear relation and its adjoint.

\begin{fact}[Br\'ezis-Browder] \emph{\cite[Theorem~2]{Brezis-Browder}}
\label{Sv:7}
Let $A\colon X \To X^*$ be a monotone linear relation
such that $\gra A$ is closed. Then the following are equivalent.
\begin{enumerate}
\item
$A$ is maximal monotone.
\item
$A^*$ is maximal monotone.
\item
$A^*$ is monotone.
\end{enumerate}
\end{fact}

The next result generalizes \cite[Proposition~3.1]{BBW}
from linear operators to linear relations.

\begin{theorem}\label{sia:3}
Let $A\colon X \To X^*$ be a maximal monotone linear relation.
Then the following hold.
\begin{enumerate}
\item \label{sia:3i}
$\ker A=\ker A^*$.
\item \label{sia:3ii}
$\overline{\ran A}=\overline{\ran A^*}$.
\item \label{sia:3iii}
$(\dom A^*)^{\bot}= A^*0=A0=(\dom A)^{\perp}$.
\item \label{sia:3iv}
$\overline{\dom A^*}=\overline{\dom A}$.
\end{enumerate}
\end{theorem}
\begin{proof} By Fact~\ref{Sv:7},
 $A^*$ is maximal monotone.

\ref{sia:3i}: Let $x\in \ker A$, $y\in \dom A$, and $\alpha\in\RR$.
Then
\begin{equation}0\leq\langle \alpha x+y, A(\alpha x+y)\rangle=\alpha^2
\langle x, Ax\rangle+\alpha\langle x, Ay\rangle+\alpha\langle y,
Ax\rangle+\langle y, Ay\rangle.\label{JC:1}\end{equation}
Since $0\in Ax$, Fact~\ref{Rea:1}\ref{Th:28} yields $Ax=A0$.
By Proposition~\ref{P:a}\ref{Th:27},
$\langle x, Ax\rangle=0$ and $\alpha\langle y, Ax\rangle=0$.
Hence, in view of \eqref{JC:1},
$0\leq \alpha \langle x, Ay\rangle+\langle y, Ay\rangle$.
It follows that $\langle x,Ay\rangle=0$. Hence
$(0,-x)\in (\gra A)^{\bot}$,
i.e., $0\in A^*x$.
Therefore, $\ker A\subset\ker A^*$. On the other hand, applying this line
of thought to $A^*$, we obtain
$\ker A^*\subset\ker A^{**}=\ker A$. Altogether, $\ker A = \ker A^*$.

\ref{sia:3ii}: Combine
\ref{sia:3i} and Fact~\ref{Rea:1}\ref{Sia:1}.

\ref{sia:3iii}:
As $A^*$ is maximal monotone, it follows from Proposition~\ref{j22} that
$(\dom A^*)^{\perp}=A^*0$.  In view of Fact~\ref{Rea:1}\ref{Sia:1}
and the maximal monotonicity of $A$, we have
$ (\dom A^*)^{\perp}=A^*0=A0=(\dom A)^{\perp}$, thus $ (\dom A^*)^{\perp}=(\dom A)^{\perp}$.

\ref{sia:3iv}:
Apply $\bot$ to \ref{sia:3iii}.
\end{proof}

\begin{corollary}
Let $A\colon X \To X^*$ be a maximal monotone linear relation
such that $\overline{\dom A}=X$. Then both $A$ and $A^*$
are single-valued and linear on their respective domains.
\end{corollary}

\begin{corollary}
Let $A\colon \RR^n\To \RR^n$ be a maximal monotone linear relation.
Then
$\kernel A=\kernel A^*$, $\ran A=\ran A^*$, and
$\dom A=\dom A^*=(A0)^{\perp}=(A^*0)^{\perp}$.
\end{corollary}

\begin{remark}
Consider Theorem~\ref{sia:3}\ref{sia:3ii}.
The Volterra operator illustrates that $\ran A$ and
$\ran A^*$ are not comparable in general (see \cite[Example~3.3]{BBW}).
Considering the inverse of the Volterra
operator, we obtain an analogous negative statement for the domain.
\end{remark}

\section{Maximal Monotone Operators with Convex Graphs}
\label{s:4}
In this section, $X$ is not assumed to be reflexive.
In \cite[Corollary~4.1]{BK}, Butnariu and Kassay discuss monotone operators with
closed convex graphs; see also \cite[Section~46]{Si}.
We shall show that if the graph
of a maximal monotone operator is convex, then the graph must in fact
be affine (i.e., a translate of a linear subspace).

\begin{proposition}\label{n:1}
Let $A\colon X \rightrightarrows X^*$ be maximal monotone such that
$\gra A$ is a convex cone. Then $\gra A$ is
a linear subspace of $X \times X^*$.
\end{proposition}

\begin{proof}
Take $(x,x^*)\in \gra A$  and also
$(y,y^*)\in \gra A$. As $\gra A$ is a convex cone, we have
$(x,x^*)+ (y,y^*)=(x+y,\ x^*+y^*)\in \gra A$.
Since $(0,0)\in \gra A$,
we obtain
$0 \leq  \langle x+y,\ x^*+y^*\rangle =
\langle(-x)-y,\ (-x^*)-y^*\rangle$.
From the maximal monotonicity of $A$, it follows that
$-(x,x^*)\in\gra A$. Therefore,
\begin{equation} \label{e:backfromVan}
-\gra A\subset \gra A.
\end{equation}
A result due to Rockafellar
(see \cite[Theorem~2.7]{Rocky}, which is stated
in Euclidean space but the proof of which works without change
in our present setting)
completes the proof.
\end{proof}

\begin{theorem}\label{openanswer}
Let $A\colon X \rightrightarrows X^*$  be maximal monotone such that $\gra
A$ is convex. Then $\gra A$ is actually affine.
\end{theorem}
\begin{proof}
Let $(x_0,x_0^*)\in \gra A $ and $B\colon X\rightrightarrows X^*$
be such that  $\gra B=\gra A - (x_0,x_0^*).$ Thus $\gra B$ is convex with $(0,0)\in \gra B$, and $B$ is  maximal monotone.
Take $\alpha \geq 0$ and $(x,x^*)\in \gra B$.
In view of Proposition~\ref{n:1},
it suffices to show that $\alpha(x,x^*)\in\gra B$.
If
$\alpha\leq 1$, then the convexity of $\gra B$ yields
$\alpha(x,x^*)=\alpha(x,x^*)+(1-\alpha)(0,0)\in \gra B$.
Thus assume that $\alpha>1$ and let $(y,y^*)\in \gra B$.
Using the  previous reasoning, we deduce that
$\tfrac{1}{\alpha}(y,y^*)\in \gra B$.
Thus,
$\langle \alpha x-y,\alpha x^*-y^*\rangle=\alpha^2\langle x-
\tfrac{1}{\alpha}y, x^*-\tfrac{1}{\alpha}y^*\rangle\geq 0$.
Since $B$ is maximal monotone,  $\alpha(x,x^*)\in \gra B$.
\end{proof}

\begin{remark}
Theorem~\ref{openanswer}
provides a complete answer to \cite[Problem~46.4 on page~183]{Si}.
Note also that for every nonzero closed proper subspace $L$ of $X$, the normal cone operator
$N_L = \partial \iota_{L}$ is a maximal monotone linear relation; however, neither $N_{L}$
nor its inverse is an affine mapping. 
\end{remark}

\section{The Fitzpatrick Function of a Sum}

\label{s:5}

Fitzpatrick functions --- introduced first by Fitzpatrick \cite{FIT} in
1988 (see also \cite{burachick,mlt,Penot2}) --- have turned out to be immensely useful in the study
of maximal monotone operators; see, e.g.,
\cite{Si} and the references therein.

\begin{definition}
Let $A\colon X\rightrightarrows X^*.$ The
\emph{Fitzpatrick function} of $A$ is
\begin{equation}
F_A\colon  (x,x^*)\mapsto \sup_{(y,y^*)\in
\gra A}\langle x, y^*\rangle+ \langle y,x^*\rangle-\langle y,y^*\rangle.
\end{equation}
\end{definition}

The following partial inf-convolution,
 introduced by Simons and Z\u{a}linescu \cite{SiZ}, plays an important role
in the study of the maximal monotonicity of the sum
of two maximal monotone operators.

\begin{definition}
Let $F_1, F_2\colon X\times X^*\rightarrow\RX$.
Then the \emph{partial inf-convolution} $F_1\Box_2 F_2$
is the function defined on $X\times X^*$ by
\begin{equation*}F_1\Box_2 F_2\colon
(x,x^*)\mapsto \inf_{y^*\in X^*}
F_1(x,x^*-y^*)+F_2(x,y^*).
\end{equation*}
\end{definition}
Let $A,B:X\To X^*$ be maximal monotone operators.
It is not hard to see that $F_{A+B}\leq F_{A}\Box_{2} F_{B}$;
moreover, equality may fail \cite[Proposition~4.2 and Example~4.7]{BMS}.
In \cite[Corollary~5.6]{BBW}, it was shown that
$F_{A+B}=F_{A}\Box_{2} F_{B}$ when $A,B$ are
continuous linear monotone operators and some constraint
qualification holds.
In this section, we substantially generalize
this result to maximal monotone linear relations.
Following \cite{Penot2}, it will be convenient to set
$ F^\intercal\colon X^*\times X\colon (x^*,x)\mapsto F(x,x^*)$,
when
$F\colon X\times X^*\to\RX$, and similarly
for a function defined on $X^*\times X$.

We start with some basic properties about Fitzpatrick functions.

\begin{proposition}\label{explicit}
Let $A\colon X\To X^*$ be monotone linear relation. Then
the following hold.
\begin{enumerate}
\item \label{expliciti}
$\gra (-A^*)=(\gra A^{-1})^{\perp}$.
\item \label{explicitii}
$F|_{\gra (-A^*)}\equiv 0$.
\end{enumerate}
\end{proposition}
\begin{proof}
\ref{expliciti}: Take $(x,x^*)\in X\times X^*$. Then
$(x,x^*)\in (\gra A^{-1})^{\perp}
\Leftrightarrow (x^*,x)\in (\gra A)^{\perp}
\Leftrightarrow (x,-x^*)\in \gra A^* \Leftrightarrow
(x,x^*)\in\gra(-A^*)$.

\ref{explicitii}: Take $(x,x^*)\in\gra(-A^*)$.
By \ref{expliciti}, $(x^*,x)\in (\gra A)^{\perp}$.
Since
$(0,0)\in \gra A$ and $A$ is monotone, we have
$F_{A}(x,x^*)\geq 0$ and $\langle y,y^*\rangle \geq 0$ for every
$(y,y^*)\in\gra A$. This yields
\begin{equation*}
F_{A}(x,x^*) =\sup_{(y,y^*)\in\gra A}\langle x^*,y\rangle +
\langle x,y^*\rangle-\langle y^*,y\rangle =
\sup_{(y,y^*)\in \gra A} 0-\langle y^*,y\rangle\leq 0.
\end{equation*}
Altogether, we have $F_{A}(x,x^*)=0$.
\end{proof}

It turns out to be convenient to define
$$P_X \colon X\times X^*\to X\colon (x,x^*)\mapsto x.$$
We shall need the following facts for later proofs.

\begin{fact}[Fitzpatrick]\label{F1}
Let $A\colon X\To X^*$ be maximal monotone.  Then
 $F_A$ is proper lower semicontinuous and convex, and
$F_A^{*\intercal} \geq
F_A \geq \scal{\cdot}{\cdot}$.
\end{fact}
\begin{proof}
See \cite[Corollary~3.9 and Proposition~4.2]{FIT}.
\end{proof}

\begin{proposition}\label{Faz:1} Let $A\colon X\To X^*$ be a monotone
linear relation such that its graph is closed. Then
$ F_A^*\colon (x^*,x)\mapsto
\iota_{\gra A^{-1}}(x^*,x) + \scal{x}{x^*}$.
\end{proposition}

\begin{proof} Define
$G\colon X^*\times X \to \RX\colon
(x^*,x)\mapsto \iota_{\gra A}(x,x^*) + \scal{x}{x^*}$.
By Proposition~\ref{P:a}\ref{Th:31}, $\scal{x}{x^*}=\scal{x}{Ax}$
for every $(x,x^*)\in\gra A$; then by Proposition~\ref{rP:1}, $G$ is  a convex function.
As  $\gra A$ is closed, $G$ is lower semicontinuous. Thus,
$G$ is a proper lower semicontinuous convex function.
By definition of $F_A$, $F_A=G^{*}$. Therefore,
we have $F_A^*=G^{**}=G$.
\end{proof}

We also set, for any set $S$ in a real vector space,
$$ \cone S := \bigcup_{\lambda>0} \lambda S = \menge{\lambda s}{\lambda>0
\text{~and~} s\in S}.$$

\begin{fact}[Simons-Z\u{a}linescu]\label{F5}
Let $A\colon X\To X^*$ be maximal monotone. Then the following hold.
\begin{enumerate}
 \item \label{F5i}
$\dom A\subset P_{X}\big(\dom  F_{A}^{*\intercal}\big)\subset P_X(\dom F_A)\subset \overline{\dom A}$.
 \item \label{F5ii}
Suppose that $A, B\colon X\To X^*$  are maximal monotone linear
relations and that $\dom A-\dom B$ is closed.
Then $A+B$ is maximal monotone.
\end{enumerate}
\end{fact}
\begin{proof}
\ref{F5i}: Combine \cite[Theorem~31.2]{Si} and \cite[Lemma~5.3(a)]{SiZ}.
\ref{F5ii}:  See \cite[Theorem~5.5]{SiZ}.
\end{proof}

\begin{fact}[Simons-Z\u{a}linescu]\label{F4}
Let $F_1, F_2\colon X\times X^* \to \RX$ be proper,
lower semicontinuous, and convex. Assume that
for every $(x,x^*)\in X\times X^*$,
\begin{equation*}(F_1\Box_2 F_2)(x,x^*)>-\infty
\end{equation*}
and that  $ \cone\big(P_X\dom F_1-P_X\dom F_2\big)$
is a closed subspace of $X$. Then for every $(x,x^*)\in X\times X^*$,
\begin{equation*}
(F_1\Box_2 F_2)^*(x^*,x)=\min_{y^*\in X^*}
F_1^*(x^*-y^*,x)+F_2^*(y^*,x).
\end{equation*}\end{fact}
\begin{proof}
See \cite[Theorem~4.2]{SiZ}.
\end{proof}

\begin{lemma}\label{Co:1}
Let $A, B\colon X\To X^*$ be  maximal monotone, and suppose that
$\cone\big(\dom A-\dom B\big)$ is a closed subspace of $X$.
Then
\begin{equation*}\cone\big(P_{X}\dom F_A-P_{X}\dom F_B\big)
=\cone\big(\dom A-\dom B\big)=
\cone\big(P_{X}\dom F^{*\intercal}_A-P_{X}\dom F^{*\intercal}_B\big).
\end{equation*}
\end{lemma}
\begin{proof}
Using Fact~\ref{F5}\ref{F5i}, we see that
\begin{equation}
\cone\big(\dom A-\dom B\big)\subset
\cone\big(P_{X}\dom F_A-P_{X}\dom F_B\big)
\subset\cone\big(\overline{\dom A}-\overline{\dom B}\big).
\label{Cor:1}
\end{equation}
On the other hand, we have
\begin{equation}
(\forall \lambda>0)\quad
\lambda \big(\overline{\dom A}-\overline{\dom B}\big)
\subset\overline{\lambda \big(\dom A-\dom B\big)}
\subset\overline{\cone\big(\dom A-\dom
B\big)}.\label{Cor:2}
\end{equation}
Thus, by \eqref{Cor:2} and the hypothesis,
\begin{equation}\cone\big(\overline{\dom
A}-\overline{\dom B}\big)
\subset\overline{\cone\big(\dom A-\dom B\big)}
=\cone\big(\dom A-\dom B\big).\label{Cor:3}
\end{equation}
Hence, by \eqref{Cor:1} and \eqref{Cor:3},
$\cone\big(P_{X}\dom F_A-P_{X}\dom F_B\big)
=\cone\big(\dom A-\dom B\big)$.
In a similar fashion,
Fact~\ref{F5}\ref{F5i} implies that
$\cone\big(P_{X}\dom F^{*\intercal}_A-P_{X}\dom
F^{*\intercal}_B\big)=
\cone\big(\dom A-\dom B\big)$.
\end{proof}

\begin{proposition}\label{F12}
Let $A,B\colon X\To X^*$  be maximal monotone and
suppose that $\cone\big(\dom A-\dom B\big)$ is a closed subspace of $X$.
Then
$ F_A\Box_2F_B$ is proper, lower semicontinuous, and convex,
and the partial infimal convolution is exact everywhere.
\end{proposition}
\begin{proof}
Take $(x,x^*)\in X\times X^*$.
By Fact~\ref{F1},
$(F_A\Box_2F_B)(x,x^*)\geq\langle x,x^*\rangle>-\infty$.
Lemma~\ref{Co:1} implies that
$$\cone\big(P_X \dom F_A-P_X\dom F_B\big)=
\cone\big(\dom A-\dom B\big)\;\; \text{is a closed subspace}.$$
Using Fact~\ref{F4}, we see that
\begin{equation}
\big(F_A\Box_2F_B\big)^{*\intercal}(x,x^*)
=\min_{y^*\in X^*}F_A^*(x^*-y^*,x)+F_B^*(y^*,x)=
\big(F^{*\intercal}_A\Box_2F^{*\intercal}_B\big) (x,x^*).
\label{Ap:3}
\end{equation}
By Fact~\ref{F1},
$$ \big(F^{*\intercal}_A\Box_2F^{*\intercal}_B\big)(x,x^*)
\geq\langle x,x^*\rangle>-\infty.$$
In view of Lemma~\ref{Co:1},
$$\cone\big(P_{X} \dom F^{*\intercal}_A-P_X\dom F^{*\intercal}_B\big)
=\cone\big(\dom A-\dom B\big)\;\; \text{is a closed subspace}.$$
Therefore, using Fact~\ref{F4} and \eqref{Ap:3},
\begin{align*}
\big(F_A\Box_2F_B\big)^{**}(x,x^*)
&=\big(F_A\Box_2F_B\big)^{*\intercal*}(x^*,x)=\min_{y^*\in X^*}{F^{*\intercal*}_A(x^*-y^*,x)+F_B^{*\intercal*}(y^*,x)}\\
&=\min_{y^*\in X^*}{F_A(x,x^*-y^*)+F_B(x,y^*)}\\
&=\big(F_A\Box_2F_B\big)(x,x^*).
\end{align*}
Hence $ F_A\Box_2F_B$ is proper, lower semicontinuous, and convex,
and the partial infimal convolution is exact.
\end{proof}

We are now ready for the main result of this section.

\begin{theorem}[Fitzpatrick function of the sum]\label{FS6}
Let $A,B\colon X\To X^*$ be maximal monotone linear relations,
and suppose that $\dom A-\dom B$ is closed.
Then $F_{A+B}= F_A\Box_2F_B$.
\end{theorem}
\begin{proof}
Lemma~\ref{Co:1} implies that
$$\cone\big(P_X \dom F_A-P_X\dom F_B\big)=
\cone\big(\dom A-\dom B\big)= \dom A-\dom B\;\;
\text{is a closed subspace}.$$
Take $(x,x^*)\in X\times X^*$.
Then, by Fact~\ref{F1},
$(F_A\Box_2F_B)(x,x^*)\geq\langle x,x^*\rangle>-\infty$.
Using Fact~\ref{F4} and Proposition~\ref{Faz:1}, we deduce that
\begin{align*}(F_A\Box_2F_B)^*(x^*,x)&=\min_{y^*\in X^*}F_A^*(x^*-y^*,x)+F_B^*(y^*,x)\\
&=\min_{y^*\in X^*}\iota_{\gra A}(x,x^*-y^*)+\langle x^*-y^*,x\rangle+\iota_{\gra B}(x,y^*)+\langle y^*,x\rangle\\
&=\iota_{\gra (A+B)}(x,x^*)+\langle x^*,x\rangle=F_{A+B}^*(x^*,x).
\end{align*}
Taking Fenchel conjugates and applying Proposition~\ref{F12}
now yields the result.
\end{proof}

\section{Maximal Monotonicity}

\label{s:6}

In this section, we shall obtain criteria for maximal monotonicity of
linear relations.
These criteria generalize some of
the results by Phelps and Simons \cite{PheSim} (or \cite[Theorem~47.1]{Si}) which
also form the base of our proofs.
The following concept of the halo (see \cite[Definition~2.2]{PheSim})
is very useful.

\begin{definition}
Let $A\colon X \To X^*$ be a monotone linear relation.
A vector $x\in X$ belongs to the \emph{halo} of $A$,
written
\begin{equation} \label{H:1}
x\in \halo A \quad \Leftrightarrow \quad
\big(\exi M\geq 0\big)\big(\forall (y,y^*)\in \gra A\big)\;\;
\langle y^*, x-y\rangle\leq M\|x-y\|.
\end{equation}
\end{definition}

\begin{proposition}\label{linear}
Let $A\colon X\To X^*$ be a monotone linear relation. Then
$\dom A \subset \halo A \subset (A0)^{\perp}$.
\end{proposition}
\begin{proof}
The left inclusion follows from the monotonicity of $A$,
while the right inclusion is seen to be true by taking
$y=0$ in \eqref{H:1}.
\end{proof}

The next two results generalize Phelps and Simons'
\cite[Lemma~2.3 and Theorem~2.5]{PheSim}; we follow  their proofs.

\begin{proposition}\label{rH:4}
Let $A\colon X \rightrightarrows X^*$ be a monotone linear relation.
Then
\begin{equation*}
\halo A = P_X \bigg( \bigcup_{\text{$B$ is a monotone extension of $A$}}
\gra B\bigg).
\end{equation*}
\end{proposition}
\begin{proof}
``$\Leftarrow$'': Let $(x,x^*)\in X\times X^*$ belong
to some monotone extension of $A$. Then
$$\big(\forall (y,y^*)\in\gra A\big)\quad
\langle y^*,x-y\rangle\leq\langle x^*,x-y\rangle\leq\|x^*\|\|x-y\|.$$
Hence \eqref{H:1} holds with $M=\|x^*\|$.

``$\Rightarrow$'': Take $x\in\halo A$. Then
there exists $M\geq 0$ such that
\begin{equation}
\big(\forall (y,y^*)\in\gra A\big)\quad
\langle y^*, x-y\rangle\leq M\|x-y\|.\label{H:2}
\end{equation}
Now set
\begin{equation*}
C:=\big\{(y,\lambda)\in X\times \RR \mid \lambda\geq M\|x-y\|\big\}
\end{equation*}
and
\begin{equation*}
D:=\big\{(y,\lambda)\in (\dom A)\times \RR \mid \lambda\leq \langle
Ay,x-y\rangle\big\}.
\end{equation*}
(Note that $\langle
Ay,x-y\rangle$ is single-valued by Proposition~\ref{P:a}\ref{Th:32} and Proposition~\ref{linear}.) 
Clearly, $C$ is convex with nonempty interior.
Proposition~\ref{P:a}\ref{Th:32+},
Proposition~\ref{rP:1}, and
Proposition~\ref{linear} imply that $D$ is convex and nonempty.
By \eqref{H:2}, $(\inte C)\cap D=\emp$.
The Separation Theorem guarantees the existence of
$\alpha\in \RR$ and of $(x^*, \mu)\in X\times\RR$
such that $(x^*, \mu)\neq (0,0)$ and
\begin{align}
\big(\forall(y,\lambda)\in C\big) \quad \label{H:5}
&\langle y,x^*\rangle+\lambda \mu\geq\alpha \\
\big(\forall(y,\lambda)\in D\big) \quad \label{H:4}
&\langle y,x^*\rangle+\lambda \mu\leq\alpha.
\end{align}
Since $(\forall \lambda\leq 0)$ $(0,\lambda)\in D$ by
Proposition~\ref{linear}, \eqref{H:4} implies that $\mu\geq 0$.
If $\mu=0$, then \eqref{H:5} yields
$\inf \langle x^*,X\rangle\geq\alpha$, which implies $x^*=0$
and hence $(x^*,\mu)= (0,0)$, a contradiction. Therefore, $\mu>0$.
Dividing the inequalities \eqref{H:5} and \eqref{H:4} by $\mu$ thus
yields
\begin{equation*}
\langle\tfrac{x^*}{\mu},x\rangle\geq\tfrac{\alpha}{\mu}
\quad\text{and}\quad
\big(\forall y\in\dom A\big)\;\;
\langle\tfrac{x^*}{\mu},y\rangle+ \langle Ay,x-y\rangle
\leq\tfrac{\alpha}{\mu}.
\end{equation*}
Therefore,
$$\big(\forall (y,y^*)\in\gra A\big) \quad
\langle\tfrac{x^*}{\mu}-y^*,x-y\rangle \geq 0.$$
Hence $(x,\tfrac{x^*}{\mu})$ is  monotonically related to $\gra A$.
\end{proof}

\begin{remark} It is interesting to note that
Proposition~\ref{rH:4} can also be proved by Simons' $M$-technique.
To see this, take $x\in \halo A$
and denote the dual closed unit ball by $B^*$. Then
$x$ is characterized by
$\inf_{y\in\dom A}M\|x-y\|-\langle Ay,x-y\rangle\geq 0$;
equivalently, by
$$\inf_{y\in \dom A}\max_{b^*\in M {B}^*}\langle b^*,x-y\rangle-\langle Ay,x-y\rangle\geq 0.$$
As the function $(y, b^*)\mapsto \langle b^*,x-y\rangle-\langle Ay,x-y\rangle$ is convex in $y$, and
concave and upper semicontinuous in $b^*$, the Minimax Theorem
\cite[Theorem~3.2]{Si} results in
$$\max_{b^*\in M{B}^*}\inf_{y\in\dom A}\langle b^*,x-y\rangle-\langle Ay,x-y\rangle=\inf_{y\in \dom A}\max_{b^*\in M {B}^*}\langle b^*,x-y\rangle-\langle Ay,x-y\rangle\geq 0.$$
Hence there exists $c^*\in M{B}^*$ such that
$$\inf_{y\in\dom A}\langle c^*,x-y\rangle-\langle Ay,x-y\rangle\geq 0.$$
Therefore, $(x,c^*)$ is monotonically related to $\gra A$.
\end{remark}

\begin{theorem}[maximality]\label{Siam:1}
Let $A\colon X \To X^*$ be a monotone linear relation.
Then $$\text{$A$ is maximal monotone}\quad
\Leftrightarrow \quad (\dom A)^{\bot}= A0\;\text{and}\;\halo A =\dom A.$$
\end{theorem}
\begin{proof}
``$\Rightarrow$'':
By Proposition~\ref{j22}, $(\dom A)^{\bot}=A0$.
Proposition~\ref{rH:4} yields $\dom A\subset \halo A$.
Now take $x\in \halo A$.
By Proposition~\ref{rH:4}, there exists $x^*$ such that $(x,x^*)$
is monotonically related to $\gra A$. Since $A$ is maximal monotone,
$(x,x^*)\in\gra A$, so $x\in \dom A$. Thus, $\halo A =\dom A$.

``$\Leftarrow$'': Suppose  $(x,x^*)\in X\times X^*$ is
monotonically related to $A$.
By Proposition~\ref{rH:4}, $x\in \halo A$.
Thus $x\in \dom A$ and we pick $x_{0}^*\in Ax$.
By Proposition~\ref{P:a}\ref{Th:31+}
and Fact~\ref{Rea:1}\ref{Th:28}, we have $x^*\in x_{0}^{*}+(\dom A)^{\perp}
=x_{0}^*+A0= Ax$. Therefore, $A$ is maximal monotone.
\end{proof}

\begin{corollary}
Let $A\colon X \To X^*$ be a monotone linear relation, and suppose
that $\dom A$ is closed. Then
$A$ is maximal monotone $\Leftrightarrow$
$(\dom A)^{\bot}= A0$.
\end{corollary}
\begin{proof}
``$\Rightarrow$'': Apply Theorem~\ref{Siam:1}.
``$\Leftarrow$'': Since $\dom A$ is closed, the hypothesis yields
$\dom A= (A0)^{\perp}$. By Proposition~\ref{linear},
$\dom A=\halo A$. Once again, apply Theorem~\ref{Siam:1}.
\end{proof}

\begin{corollary}
Let $A\colon \RR^n\To\RR^n$ be a monotone linear relation.
Then $A$ is maximal monotone $\Leftrightarrow$ $(\dom A)^\bot = A0$.
\end{corollary}

\section{Characterization of Skew Monotone Linear Relations}
\label{s:7}

As an application of Theorem~\ref{Siam:1}, we shall
characterize skew linear relations.
Theorem~\ref{GF:4} below extends
\cite[Theorem~2.9]{BBW}
from monotone linear operators
to monotone linear relations.

\begin{definition}[skew linear relation]
Let  $A\colon X \To X^*$ be  a linear relation.
We say that $A$ is \emph{skew} if $A^*=-A$.
\end{definition}

\begin{proposition}
Let $A\colon X\To X^*$ be a skew linear relation.
Then both $A$ and $A^*$ are maximal monotone.
\end{proposition}
\begin{proof}
By Fact~\ref{Rea:1}\ref{Sia:2b},
$(\forall x\in\dom A)$ $\langle Ax,x\rangle=0$.
Thus, using Proposition~\ref{P:a}\ref{Th:31} and
Fact~\ref{Rea:1}\ref{Sia:2b},  we see that both
$A$ and $A^*$ are monotone.
By Fact~\ref{Sv:7} and Fact~\ref{Rea:1}\ref{Sia:1},
$A$ and $A^*$ are maximal monotone.
\end{proof}

\begin{definition}[Fitzpatrick family]
Let  $A\colon X \To X^*$ be a maximal monotone linear relation.
The associated \emph{Fitzpatrick family}
$\mathcal{F}_A$ consists of all functions $F\colon
X\times X^*\to\RX$ that are lower semicontinuous and convex,
and that satisfy
$F\geq \scal{\cdot}{\cdot} $, and $F=\scal{\cdot}{\cdot}$ on $\gra A$.
\end{definition}

\begin{fact}[Fitzpatrick]\label{GF:2}
Let  $A\colon X \To X^*$ be a maximal monotone linear relation.
Then for every $(x,x^*)\in X\times X^*$,
\begin{equation}
F_A(x,x^*) = \min\menge{F(x,x^*)}{F\in \mathcal{F}_A}
\quad\text{and}\quad
F_A^{*\intercal}(x,x^*) = \max\menge{F(x,x^*)}{F\in \mathcal{F}_A}.
\end{equation}
\end{fact}
\begin{proof}
See \cite[Theorem~3.10]{FIT}.
\end{proof}

\begin{example}\label{GF:3}
Let $A\colon X\To X^*$ be a skew linear relation.
Then
$F_A=F_{A}^{*\intercal}=\iota_{\gra A}$.
\end{example}
\begin{proof}
Since $(\forall x\in\dom A)$ $\langle Ax,x\rangle=0$,
Proposition~\ref{Faz:1} implies that
$F_A^{*\intercal}=\iota_{\gra A}$.
Moreover,
$$F_{A}=\big(F_{A}^{*\intercal}\big)^{*\intercal}=
\big(\iota_{\gra A}\big)^{*\intercal}=
\big(\iota_{\gra A}^{\intercal}\big)^*=
\big(\iota_{\gra A^{-1}}\big)^*=\iota_{(\gra A^{-1})^{\perp}}
=\iota_{\gra (-A^*)}=\iota_{\gra A},$$
by Proposition~\ref{explicit}\ref{expliciti}.
Therefore, $F_A=F_{A}^{*\intercal}=\iota_{\gra A}$.
\end{proof}

We now characterize skew linear relations
in terms of the Fitzpatrick family.
Note that the Fitzpatrick family is in this case
as small as possible, i.e., a singleton.
(For a related discussion concerning subdifferential
operators, see \cite[Section~5]{BBBRW}.)

\begin{theorem}\label{GF:4}
Let $A\colon X\To X^*$ be a maximal monotone linear relation.
Then
\begin{equation*}
\text{$A$ is skew \; $\Leftrightarrow$ \;
$\dom A=\dom A^*$ and $\mathcal{F}_A$ is a singleton,}
\end{equation*}
in which case $\mathcal{F}_A = \{\iota_{\gra A}\}$.
\end{theorem}
\begin{proof}
``$\Rightarrow$'': Combine Example~\ref{GF:3} with Fact~\ref{GF:2}.

``$\Leftarrow$'':
Fact~\ref{GF:2} and Proposition~\ref{Faz:1} yield
\begin{equation}\label{lower=upper}
F_A=F_A^{*\intercal}=\iota_{\gra A}+\langle \cdot,\cdot\rangle.
\end{equation}
By Proposition~\ref{explicit}\ref{explicitii},
for every $(y,y^*)\in \gra(-A^*)$, we have $F_{A}(y,y^*)=0$;
hence, in view of \eqref{lower=upper},
$ (y,y^*)\in \gra A$ and $\langle y,y^*\rangle=0$.
Thus
\begin{equation}\label{inclusion}
\gra -A^*\subset \gra A\quad \text{ and }\quad
\big(\forall (y,y^*)\in\gra A^*\big)\;\;
\langle y^*, y\rangle =0.
\end{equation}
Since $A$ is monotone (by hypothesis), so is $-A^*$.
We wish to show that $-A^*$ is maximal monotone.
To this end, take $x\in \halo(-A^*)$.
According to Proposition~\ref{rH:4},
there exist $x^*\in X^*$ such that $(x,x^*)$ is
monotonically related to $\gra(-A^*)$, i.e.,
$(\forall (y,y^*)\in\gra A^*)$
$\langle x-y,x^*+y^*\rangle\geq 0$; equivalently,
\begin{equation}\label{simple}
\big(\forall (y,y^*)\in\gra A^*\big) \quad
\langle x^*,x\rangle +\langle y^*,x\rangle -\langle x^*,y\rangle-\langle y^*,y\rangle\geq 0.
\end{equation}
Using \eqref{inclusion}, this in turn is equivalent to
$$\big(\forall (y,y^*)\in\gra A^*\big)\quad
\langle x^*,x\rangle\geq -\langle y^*,x\rangle +\langle x^*,y\rangle, $$
and --- since $\gra A^*$ is a linear subspace of $X\times X^*$ --- also to
$$\big(\forall (y,y^*)\in\gra A^*\big)\quad
0=-\langle y^*,x\rangle +\langle x^*,y\rangle=
\langle(x^*,-x),(y, y^*)\rangle.$$
Thus,
$(x,x^*)\in \gra A^{**}=\gra A$ (Fact~\ref{Rea:1}\ref{Sia:1}) and
in particular $x\in \dom A$. As $\dom A=\dom A^*$, we have
$x\in\dom A^*=\dom(-A^*)$. Therefore,
$ \halo(-A^*)\subset \dom(-A^*)$. The opposite inclusion is clear
from Proposition~\ref{linear}. Altogether,
 \begin{equation}\label{halodom}
\dom(-A^*)=\halo(-A^*).
 \end{equation}
By Fact~\ref{Sv:7}, $A^*$ is maximal monotone; hence,
Theorem~\ref{Siam:1} yields
$(\dom A^*)^{\bot}=A^*0$.
Since $\dom A^*=\dom(-A^*)$ and $A^*0=-A^*0$, we have
\begin{equation}\label{doma0}
\big(\dom(-A^*)\big)^{\bot}=-A^*0.
\end{equation}
Using \eqref{doma0}, \eqref{halodom}, and Theorem~\ref{Siam:1}, we conclude that $-A^*$ is maximal monotone.
Since $A$ is maximal monotone,  the inclusion in \eqref{inclusion}
implies that $A=-A^*$. Therefore, $A$ is skew.
\end{proof}

\section*{Acknowledgment}

Heinz Bauschke was partially supported by the Natural Sciences and
Engineering Research Council of Canada and
by the Canada Research Chair Program.
Xianfu Wang was partially supported by the Natural
Sciences and Engineering Research Council of Canada.


\newpage


\small

\begin{thebibliography}{99}

\bibitem{arens}
R.\ Arens, ``Operation calculus on linear relations'',
\emph{Pacific Journal of Mathematics} 19 (1961), pp.~9--23.

 \bibitem{aubin}
J.-P.\ Aubin and H.\ Frankowska, \emph{Set-Valued Analysis},
Birkh\"auser, 1990.

\bibitem{BBBRW}
S.\ Bartz, H.\ H.\ Bauschke, J.\ M.\ Borwein, S.\ Reich, and X.\ Wang,
``Fitzpatrick functions, cyclic monotonicity and Rockafellar's
antiderivative'',
\emph{Nonlinear Analysis} 66 (2007), pp.~1198--1223.

\bibitem{BBW}
H.\ H.\ Bauschke, J.\ M.\ Borwein, and X.\ Wang,
``Fitzpatrick functions and continuous linear monotone operators'',
\emph{SIAM Journal on Optimization} 18 (2007), pp.~789--809.

\bibitem{BMS}
H.\ H.\ Bauschke, D.\ A.\ McLaren, and H.\ S.\ Sendov,
``Fitzpatrick functions: inequalities, examples, and remarks on
a problem by S.\ Fitzpatrick'',
\emph{Journal of Convex Analysis} 13 (2006), pp.~499--523.


\bibitem{borwein1} J.\ M.\ Borwein, ``Adjoint process duality'',
\emph{Mathematics Operation Research}
8 (1983), pp.~403--434.

\bibitem{Brezis-Browder}
H.\ Br\'{e}zis and F.\ E.\ Browder,
``Linear maximal monotone operators
and singular nonlinear integral equations of Hammerstein
type'', in \emph{Nonlinear analysis (collection of papers in honor of
Erich H.\ Rothe)}, Academic Press, 1978, pp.~31--42.

\bibitem{burachick} R.\ S.\ Burachik and B.\ F.\ Svaiter,
``Maximal monotone operators, convex functions and
a special family of enlargements'',
\emph{Set-Valued Analysis} 10 (2002),  pp.~297--316.

\bibitem{BK}
D.\ Butnariu and G.\ Kassay,
``A proximal-projection method for finding zeros of
set-valued operators'', preprint, 2007.

\bibitem{Cross}
R.\ Cross,
\emph{Multivalued Linear Operators},
Marcel Dekker, 1998.

\bibitem{FIT}
S.\ Fitzpatrick, ``Representing monotone operators by convex
functions'', in  \emph{Workshop/Miniconference on Functional Analysis
and Optimization (Canberra 1988)}, Proceedings of the Centre for
Mathematical Analysis, Australian National University vol.~20,
Canberra, Australia,  1988, pp.~59--65.

\bibitem{mlt}
J.-E.\ Mart\'{\i}nez-Legaz and M.\ Th\'era,
``A convex representation of maximal monotone operators'',
\emph{Journal of Nonlinear and Convex Analysis} 2 (2001), pp.~243--247.

\bibitem{Penot2}
J.-P.\ Penot,
``The relevance of convex analysis for the study of monotonicity'',
\emph{Nonlinear Analysis} 58 (2004), pp.~855--871.






\bibitem{ph}
R.\ R.\ Phelps,
\emph{Convex functions, Monotone Operators and
Differentiability}, 
Springer-Verlag, 1993.

\bibitem{PheSim}
R.\ R.\ Phelps and S.\ Simons,
``Unbounded linear monotone operators on nonreflexive
Banach spaces'',
\emph{Journal of Convex Analysis} 5 (1998), pp.~303--328.





\bibitem{Rocky}
R.\ T.\ Rockafellar, \emph{Convex Analysis},
Princeton University Press, 1970.

\bibitem{RockWets}
R.\ T.\ Rockafellar and R.\ J-B Wets,
\emph{Variational Analysis},
Springer-Verlag, 1998.




\bibitem{Si}
S.\  Simons,
\emph{From Hahn-Banach to Monotonicity},
Springer-Verlag, 2008.

\bibitem{SiZ}
S.\ Simons and C.\  Z{\u{a}}linescu,
``Fenchel duality, Fitzpatrick functions and
maximal monotonicity'',
\emph{Journal of Nonlinear and Convex Analysis}
6 (2005), pp.~1--22.

\bibitem{Zalinescu}{C.\ Z\u{a}linescu},
\emph{Convex Analysis in General Vector Spaces},
World Scientific Publishing, 2002.
\end{thebibliography}
\end{document}